\documentclass{amsart}
\usepackage{amssymb,amsfonts}
\usepackage[all,arc]{xy}
\usepackage{enumerate}
\usepackage{color,soul}
    \sethlcolor{white}
\usepackage{mathrsfs}
\usepackage{pinlabel}
\usepackage{anyfontsize}
\usepackage{t1enc}
\usepackage{rotating}
\usepackage{thmtools}
\usepackage{thm-restate}

\usepackage{hyperref}

\usepackage{cleveref}

\declaretheorem[name=Theorem,numberwithin=section]{thm}
\usepackage{xparse}
\usepackage{graphicx}
\usepackage{caption}

\usepackage{subcaption,graphicx}
\theoremstyle{definition}

\newtheorem{cor}[thm]{Corollary}
\newtheorem{conv}[thm]{Convention}
\newtheorem{prop}[thm]{Proposition}
\newtheorem{lem}[thm]{Lemma}
\newtheorem{conj}[thm]{Conjecture}
\newtheorem{quest}[thm]{Question}
\usepackage{graphicx}
\newtheorem{defn}[thm]{Definition}

\theoremstyle{remark}
\newtheorem{rem}[thm]{Remark}

\makeatletter
\let\c@equation\c@thm
\makeatother
\numberwithin{equation}{section}

\title{Widths of links via diagram colorings}

\author{Ricky Lee, Puttipong Pongtanapaisan, and Hanh Vo}

\begin{document}

 \begin{abstract}
In this paper, we define invariants of links in terms of colorings of link diagrams and prove that these invariants coincide with various notions of widths of links with respect to the standard Morse function. Our formulations are advantageous because they are algorithmic and suitable for program implementations. As an application, we calculate the max-width of over 10000 links up to 14 crossings from the link table.
\end{abstract}

\maketitle
\section{Introduction}

Several important geometric invariants are extracted from a link in a standard Morse position by analyzing how the critical points are arranged. Many authors have demonstrated that when the minimum value of such an invariant is attained, useful conclusions can be drawn about the geometry and topology of the link exterior \cite{taylor2013c,scharlemann20063,gabai1987foliations}. For instance, a crucial step in the proofs of the property R conjecture by Gabai \cite{gabai1987foliations} and the knot complement conjecture by Gordon-Luecke \cite{mca1989knots} is to put a knot in a width-minimizing position. Another famous application is due to Thompson \cite{thompson1997thin}, where she showed that if a width-minimizing position of a link in the 3-sphere is not a bridge position position then the link complement contains an essential meridional planar surface.

Hayashi and Shimokawa generalized Thompson's result to links in more general 3-manifolds \cite{hayashi2001thin} (see also \cite{blair2019height}). Their definition of width is a list rather than a single number. Recording the width this way has the advantage that the Gabai width can be extracted by adding up the entries in the list in a natural fashion and Ozawa's trunk can be obtained by taking the maximum integer in the list \footnote{These notions also appeared in Lackenby's paper \cite{lackenby2010spectral} as lex-width, sum-width, and max-width}. Another motivation for studying Ozawa's trunk comes from studying polymers in nanochannels \cite{beaton2018characterising,beaton2022entanglement}. More precisely, Arsuaga et al. showed that a link in a lattice tube, which can be modeled a confined polymer, fits in an $(N\times M)$-tube if and only if the trunk of the link is strictly less than $(N+1)(M+1)$ \cite{ishihara2017bounds}.

Like many other invariants, one can get an upper bound for this lexicographical width by looking at the maxima and the minima of a link diagram from a link table. However, the estimate obtain this way is usually not very accurate. For example, it is known that certain invariants are never obtainable on a crossing number minimizing diagram \cite{blair2019incompatibility}. The goal of this paper is to provide an alternative method of calculating invariants related to Morse positions by coloring link diagrams.

The idea to use colorings to dictate embeddings was pioneered by Blair et al. \cite{blair2020wirtinger}. They used this technique to verify the Meridional Rank Conjecture for at least 1, 363,137, or approximately 80.1\%, of all 1,701,936 prime knots with crossing number at most 16 \cite{blair2022coxeter}. Their method was adapted by the first author to compute Gabai widths of knots in \cite{lee2019algorithmic} and by the second author to compute virtual bridge numbers in \cite{pongtanapaisan2019wirtinger}. The following theorem is the main result of this paper.

\begin{restatable}{thm}{main}
\label{thm:main}
    Let $L$ be a link in $S^3.$ Then, the lex-width of $L$ is equal to the Wirtinger lex-width. 
\end{restatable}

Theorem \ref{thm:main} also holds if lex-width is replaced by max-width or sum-width. This generalizes the main result of \cite{lee2019algorithmic} from one component to multiple components, and from one notion of width to three flavors. Applications of our theorem are listed in the final section. For instance, using our Python code, we calculated the widths precisely for numerous links up to 14 crossings from the link table. For these links, we can use the main result in \cite{ishihara2017bounds} to conclude that they can be embedded in a $(3\times 1)$-lattice tube. Via 2-fold branched covering, we also obtain the values of the width of some 3-manifolds.

\subsection*{Organizations}
The paper is organized as follows. In Section \ref{section:prelims} we discussed basic terminologies related to widths and thin positions. We also include important propositions that we will used later in the paper that were already proved in the literature. In Section \ref{sec:coloringfromembedding}, we prove that the lex-width define in terms of coloring in this paper is at most the lex-width defined in its original conception. Theorem \ref{thm:main} is proved in Section \ref{sec:embeddingfromcoloring}, where the other inequality comparing the two ways of defining the lex-width is proved. Section \ref{sec:computational} is devoted to explaining how the implementation using Python works. In Section \ref{sec:apps}, we finish the paper with some applications and questions.

\section{Preliminaries}\label{section:prelims}

In this section, we review the concepts that previously appeared in the literature.
\subsection{Original formulations of widths}
Let $L$ be an ambient isotopy class of a link in $S^3$ and let $h : S^3 \rightarrow \mathbb{R}$ be the standard height function. Let $\gamma$ be a smooth embedding of link type $L$, where $h|_{\gamma}$ is Morse and all critical points of $h|_{\gamma}$ have distinct critical values. We define the \textit{lex-width} of $\gamma$ to be the multi-set of intersection numbers of the regular level surfaces. Here, one re-orders each multi-set into a descending sequence. Since the value only changes as one passes through a critical point, we need to only record one number for each regular level adjacent to critical levels. 
\begin{defn}
    Using the lexicographical ordering of multi-sets, we define the \textit{lex-width} lex$(L)$ of $L$ to be the minimum lex-width over all embeddings of a link type $L$. 
\end{defn}

Next, from the multi-set of intersection numbers we used to define the lex-width, we define the \textit{Gabai width} to be the sum of the numbers in the multi-set.

\begin{defn}
    The \textit{Gabai width} $w(L)$ of $L$ is the minimum Gabai width over all embeddings of a link type $L$. 
\end{defn}

Lastly, from the multi-set of intersection numbers we used to define the lex-width, we define the \textit{trunk} to be the maximum element from the multi-set.

\begin{defn}
    The \textit{trunk} of $L$, denoted by  trunk$(L)$ is the minimum trunk over all embeddings of a link type $L$. 
\end{defn}

For example, the embedding of a link $L$ in Figure \ref{trunkexample} has lex-width $\{6,6,4,4,4,2,2\}.$ The Gabai width of that embedding is then the sum of the terms, which is 28. The trunk of that embedding is 6, which is the highest term in the sequence. In fact, one can show that the lex-width, Gabai width, and trunk of $L$ are precisely $\{6,6,4,4,4,2,2\}, 28,$ and 6, respectively. 

The following terminologies are commonly used.

\begin{defn}
    We say that $S$ is a \textit{thin level} (resp. \textit{thick level}) for a link $L$ with respect to the standard height function $h$ if $S = h^{-1}(t)$ for some $t$ which lies between adjacent critical values $x$ and $y$ of $h$, where $x$ is a minimum (resp. maximum) of $L$ lying above $t$ and $y$ is a maximum (resp. minimum) of $L$ lying below $t$.
\end{defn}
\begin{rem}
      It is not known whether minimizing one notion of widths mentioned above always gives a position, where the other notions of widths are also minimal. For instance, it is conceivable that in a Gabai width minimizing position thin position, the thickest level does not give the trunk of $L$. Potential examples were proposed in \cite{davies2017natural}.
\end{rem}

\begin{figure}[ht!]
\centering
\includegraphics[width=4cm]{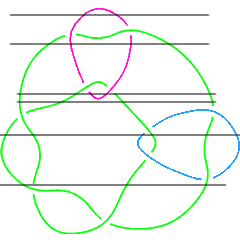}
\caption{One can visualize the intersection of the level planes and the link by parallel lines in the link diagram. This figure shows that the link L11a496 has an embedding with trunk equals six.}\label{trunkexample}
\end{figure}

\subsection{Wirtinger colorings}

In this subsection, we review the terminologies that were used to study the bridge number and Gabai width of knots via diagram colorings.

Treating a link diagram as a disjoint union of arcs, a \textit{strand} is a connected component of a diagram. Given a link diagram, where some strands are colored, a \textit{coloring move} is an addition of one more color at a crossing according to the following rule: if an overstrand is colored, and the incoming strand has the color $c$, then the outgoing strand receives the color $c$ after the coloring move (see Figure \ref{fig:coloringmove}).

\begin{figure}[ht!]
\centering
\includegraphics[width=4cm]{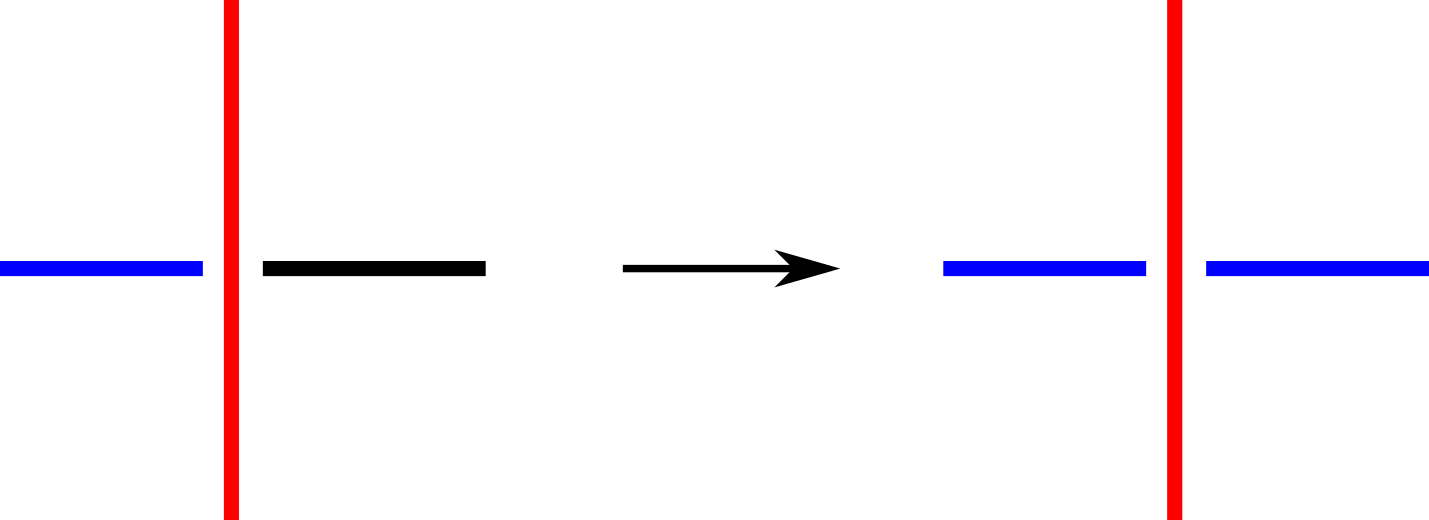}
\caption{The coloring move}\label{fig:coloringmove}
\end{figure}

The strands that we assign the colors to in the beginning prior to coloring moves are called \textit{seeds}. 

A coloring can be specified by a tuple $(A,f)$, where $A$ is a subset of the set of all strands (the colored strands) and where $f:A\rightarrow \mathbb{Z}$ (i.e., which strand receives which color).

\begin{defn}
If $(A_0, f_0) \rightarrow \cdots \rightarrow (A_t, f_t)$ is a sequence of coloring moves and seed additions on
$D$, then we say the sequence is a \textit{partial coloring sequence}. If we have a partial coloring sequence $(A_0, f_0) \rightarrow \cdots \rightarrow (A_t, f_t)$ such that $s(D) = A_t$, then we say the sequence is a \textit{completed coloring
sequence}. If $t$ is an index of a partial coloring $(A_t, f_t)$ in a specified coloring sequence, we refer to $t$ as a \textit{stage}.
\end{defn}

Denote by $s(D)$ the set of strands. If $s \in s(D)$, then the crossings at the two endpoints of $s$ will be referred to as the crossings \textit{incident} to $s$. If $s_p$ and
$s_q$ are the understrands of the same crossing, then we say $s_p$ and $s_q$ are \textit{adjacent}. If $\{s_1, s_2, \cdots , s_m\} \subset
s(D)$ is a subset of strands such that $s_i$
is adjacent to $s_{i+1}$ for each $i \in \{1, 2, \cdots , m-1\}$, then we say
the set $\{s_1, s_2, \cdots , s_m\}$ is \textit{connected}.

We also remind the reader of an important property that we will frequently use. This was essentially Proposition 2.2 in \cite{blair2020wirtinger}.

\begin{prop}\label{prop:connected}
    At every stage of the coloring process, each color in the diagram corresponds to a connected arc of $L$.
\end{prop}
The coloring algorithm in \cite{blair2020wirtinger} requires all assignments of colors to happen in the beginning and then all coloring moves are performed. In order to study widths, we will allow another move that adds a new color not used for the seeds to a strand later on in the process. This move is called \textit{seed addition} \cite{lee2019algorithmic}. This is demonstrated in Figure \ref{fig:wirtwidth}, where three new colors are assigned in the beginning, then some coloring moves are performed. The seed addition (purple) is added later on.

\subsection{Cut-split links}\label{subsec:cutsplit}
In \cite{blair2020wirtinger}, there is a certain family of links that requires addition arguments in order obtain an embedding.

A link $L$ is \textit{cut-split} if there exists an unknotted component $U$ of $L$ such that $U$ bounds an embedded disk $B$ in $S^3$ with int($B
) \cap L = \emptyset$ or $|L \cap int(B)|=1$. A link diagram $D$ is \textit{cut-split} if there exists understrands $s_p,s_q$ adjacent at some crossing of $D$ such that $s_p = s_q$ or if there exists a strand that is itself a simple closed curve.

We will handle this case separately as well in this paper.
\section{Coloring from the height function}\label{sec:coloringfromembedding}

In this section, we show that the Wirtinger lex-width is at most the lex-width. With minor adaptations, we also show that the Wirtinger width is at most the Gabai width and the Wirtinger trunk is at most the trunk. Together with the next section, we have the equality between the Wirtinger versions and the original formulations.

\begin{conv}
    When we prove a statement for width with no prefix, we prove the statement for all flavors of widths (lex,max, and sum).
\end{conv}
\begin{defn}
    We say that a coloring stage contains a \textit{special crossing}, if that stage contains either
\begin{enumerate}
    \item A crossing where the overstrand is colored, and the colorings of the adjacent understrands are different (see Figure \ref{fig:wirtwidth}), or
    \item a crossing where an overstrand $a_k$ is colored and the colorings of the adjacent understrands $a_i$ and $a_j$ are the same, but $a_i$ and $a_j$ belong to a one-colored link component $C$. Furthermore, $a_j$ is the final strand of $C$ that receives the color (see Figure \ref{fig:type2}).
\end{enumerate}
We refer to a crossing that fits the description in (1) (resp. in (2)) as a crossing of \textit{type I} (resp. \textit{type II}).
\end{defn}

\begin{figure}[ht!]
\centering
\includegraphics[width=9.9cm]{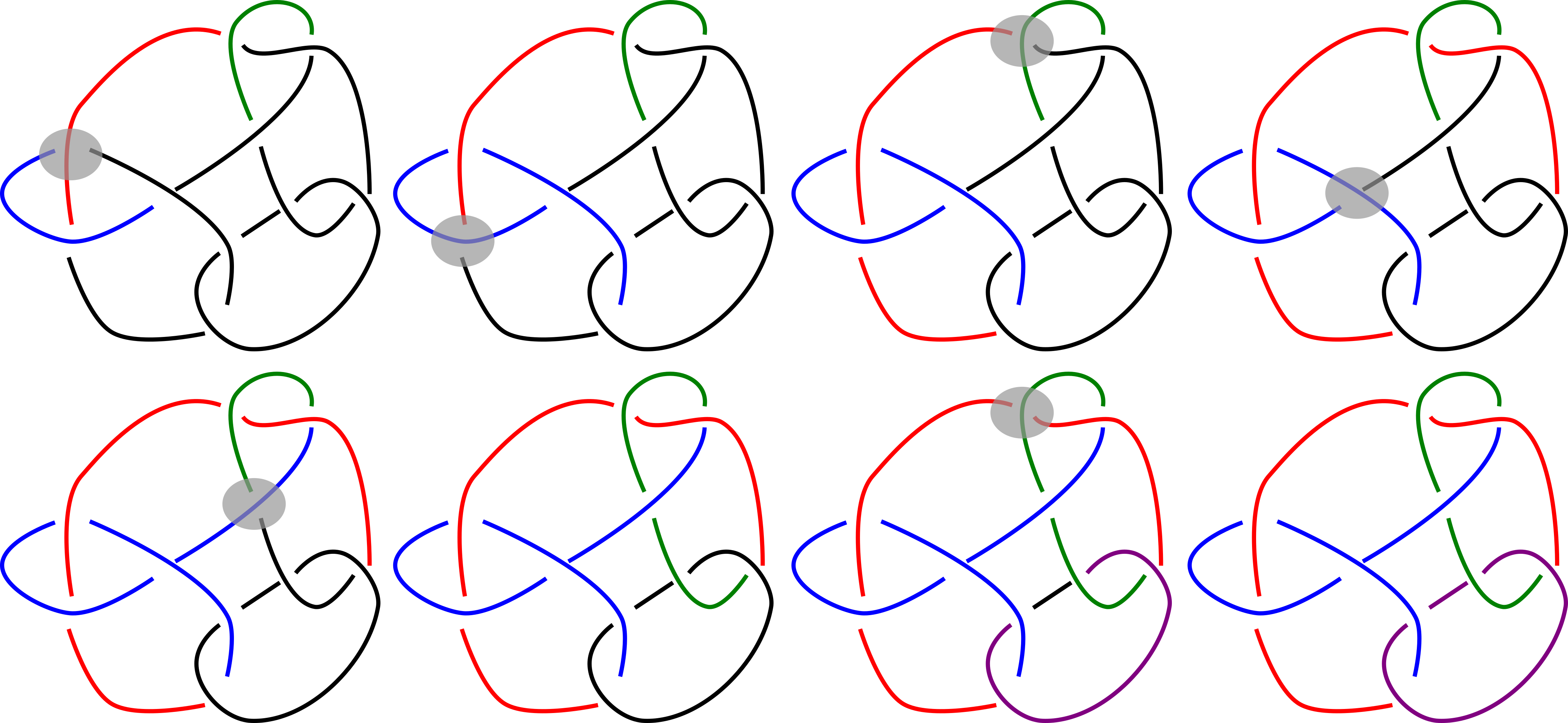}
\caption{A coloring sequence is demonstrated on L10n35 from left to right. The red, blue, and green strands generate the coloring until a special crossing of type I appears in the leftmost diagram on the second row. The purple coloring is added in the penultimate diagram of the second row. This implies that there is an embedding of L10n35 where a minimum is higher than a maximum, giving trunk(L10n35) = 6.}\label{fig:wirtwidth}
\end{figure}

\begin{figure}[ht!]
\centering
\includegraphics[width=9.7cm]{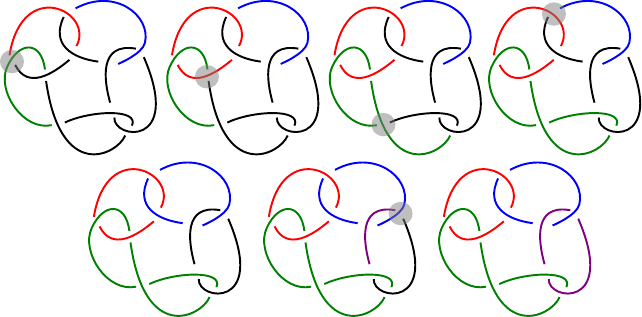}
\caption{Another coloring sequence is demonstrated on L9a55 from left to right. The red, blue, and green strands generate the coloring until a special crossing of type II appears in the leftmost diagram on the second row. The purple coloring is added in the penultimate diagram of the second row. This implies that there is an embedding of L9a55 where a minimum is higher than a maximum, giving trunk(L9a55) = 6.}\label{fig:type2}
\end{figure}

 Let $(A_0, f_0) \rightarrow \cdots \rightarrow (A_J , f_J)$ be a completed coloring sequence. From this sequence, consider the subsequence $d_j$ of seed additions and special crossings (retaining the ordering) that appear. To formally capture this, we introduce the following definition.
\begin{defn}
    Let $(A_0, f_0) \rightarrow \cdots \rightarrow (A_J , f_J )$ be a completed coloring sequence with the special
crossing set $\mathcal{C}$. Let $\mathcal{C}_t$ denote the set of crossings that become special at stage $t$. A $\Delta$-ordering is an enumeration of the elements in $s(D)\cup \mathcal{C}$ satisfying the following conditions:
\begin{enumerate}
    \item  For all $0 \leq t < u \leq J$, all colored strands or crossings that are special at stage $t$ are listed before any colored strands or crossings that are special at stage $u$.
    \item For each stage $0 \leq t \leq J$, the element in $A_t \backslash A_{t-1}$ is listed, followed by all elements in $\mathcal{C}_t$ (if
$\mathcal{C}_t \neq \emptyset)$.
\end{enumerate}

\end{defn}

\begin{defn}\label{def:wirtlexwidth}
    We define the \textit{Wirtinger lex-width} of a complete coloring sequence to be the multi-set $\{a_i\}$ satisfying the following rules:
    \begin{itemize}
        \item $a_0=2$.
        \item If $d_j$ corresponds to a seed addition, then $a_j=a_{j-1}+2$.
        \item If $d_j$ is a special crossing, then $a_j=a_{j-1}-2$.
    \end{itemize}
    The \textit{Wirtinger lex-width} $\mathbb{WL}(D)$ of a diagram is defined to be the minimum Wirtinger lex-width taken over all possible completed coloring sequences defined for the diagram $D$. The \textit{Wirtinger lex-width} $\mathbb{WL}(L)$ of a link is the minimum Wirtinger lex-width taken over all diagrams representing $L$.
\end{defn}

\begin{defn}
    We define the \textit{Wirtinger sum-width} of a complete coloring sequence to be the sum of the elements of the multi-set $\{a_i\}$ from Definition \ref{def:wirtlexwidth}. The \textit{Wirtinger sum-width} $\mathbb{WS}(D)$ of a diagram is defined to be the minimum Wirtinger sum-width taken over all possible completed coloring sequences defined for the diagram $D$. The \textit{Wirtinger sum-width} $\mathbb{WS}(L)$ of a link is the minimum Wirtinger sum-width taken over all diagrams representing $L$.
\end{defn}

\begin{defn}
    We define the \textit{Wirtinger trunk} of a complete coloring sequence to be the maximum element of the multi-set $\{a_i\}$ from Definition \ref{def:wirtlexwidth}. The \textit{Wirtinger trunk} $\mathbb{WT}(D)$ of a diagram is defined to be the minimum Wirtinger trunk taken over all possible completed coloring sequences defined for the diagram $D$. The \textit{Wirtinger trunk} $\mathbb{WT}(L)$ of a link is the minimum Wirtinger trunk taken over all diagrams representing $L$.
\end{defn}

    Given an embedding realizing the lex-width with respect to the standard height function $h(x,y,z)=z$, the goal of this section is to obtain a diagram such that a completed coloring sequence gives a desired ordering of strand colorings and special crossings. Let $D$ be the diagram obtained by projection of $L$ into the yz-plane via the projection map $p:\mathbb{R}^3\rightarrow\mathbb{R}^2$. We can view $D$ as a 4-valent graph, where the edges are strands and the vertices are crossings, with over-under labels at the vertices dictating which strands are over-strands and which are under-strands.

\subsection{Coloring by Height}
	
	We refer the the \emph{height of a strand} $s$ in $D$, denoted $h(s)$, to be sup$_{y\in s}h(y)$. Note that since $h|_L$ is Morse (by assumption, as $L$ realizes minimal  with respect to $h$), the height of a strand is well defined. Performing an arbitrarily small perturbation if necessary, we assume without loss of generality that all strands of $D$ have distinct heights. Moreover, we can also assume that the endpoints of all strands have distinct heights as well. Let $\{s_0,s_1,\ldots, s_J\}$ denote the ordering of the strands by decreasing height. 
	
	We obtain a completed coloring sequence on $D$ 
	\[	(A_0,f_0)\rightarrow \ldots \rightarrow (A_J,f_J)	\]
	satisfying the following properties:
	\begin{enumerate}
		\item Each strand $s$ where $h|_s$ attains it's maximum in the interior of $s$ receives its color via a seed addition.
		\item For all $t>0$, we have $A_t\setminus A_{t-1}=\{s_t\}$ and $A_0=\{s_0\}$. That is, we color the strands by decreasing height.
	\end{enumerate}
	An induction argument can be used to show that such a completed coloring sequence can indeed be obtained. A completed coloring sequence obtained in this way is said to be obtained via \emph{coloring by height}. 
    
To guarantee that before we arrive at a thin level, the number of local minima is not more than the number of special crossings, we will need to appeal to the following lemma. The first author proved the analog of this for Gabai width (see Lemma 5.4 in \cite{lee2019algorithmic}).
\begin{lem}\label{lem:nowiggle}
    Suppose that an embedding of $L$ realizes the lex-width, Gabai width, or max-width. Consider its projection to the $yz$-plane. Let $s(D)$ be the set of strands of the resulting projection. If $s\in s(D)$ and $l\in \mathbb{R}^2$ is a line, which is a projection of a regular value, then $|l \cap s|\leq 2$.
\end{lem}
\begin{proof}
    Suppose for contradiction that there is a regular level contributing $|l \cap s| > 2$. We will find another diagram for $L$ with strictly lower lex-width. Take three consecutive points $a,b,c$ of $|l \cap s|$. Let $s_{ab}$ denote the subarc of $s$ connecting $a$ to $b.$ The arc $s_{bc}$ is defined similarly.

    In any case, $p^{-1}(l)$ divides $\mathbb{R}^3$ into two sides $R_+,$ and $R_-$. Furthermore, there is a subarc $p^{-1}(s_{ab})$ of $L$ that cobounds a disk with $p^{-1}(l)$ in $R_+,$ and there is a subarc $p^{-1}(s_{bc})$ of $L$ that cobounds a disk with $p^{-1}(l)$ in $R_-$. An isotopy can then be performed, which removes two intersection points of $p^{-1}(l)$ with $L$ (see Figure \ref{fig:decompose}). More details on the existence of such disks can be found in  in \cite{lee2019algorithmic}, where it was argued that this reduces the Gabai width. We now argue that it reduces the max-width and lex-width as well.

    Observe that twice the max-width is at most the bridge number. Since the isotopy decreases the bridge number by one, it decreases the max-width by 2. 

    To see that the lex-width also decreases after the move, we observe that the sequence used to calculate the lex-width has two fewer terms and changes from 
    \begin{center}
            $\lbrace 2,\cdots ,a_i-2,a_i,a_{i}+2,a_{i}, \cdots, 2\rbrace$
    \end{center}
to a sequence $\lbrace 2,\cdots ,a_i-2,a_i,a_{i}-2, \cdots, 2\rbrace$, which is strictly less after rearranging the terms in decreasing order. More specifically, suppose that $S$ and $S'$ are the decreasing sequences before and after the move, respectively. Then, there is a smallest index $j$ in the sequence, where the $j$-th term in $S$ is $a_i+2$, and the $j$-th term in $S$ is $a_i.$
\end{proof}

\begin{figure}[ht!]
\centering
\includegraphics[width=5cm]{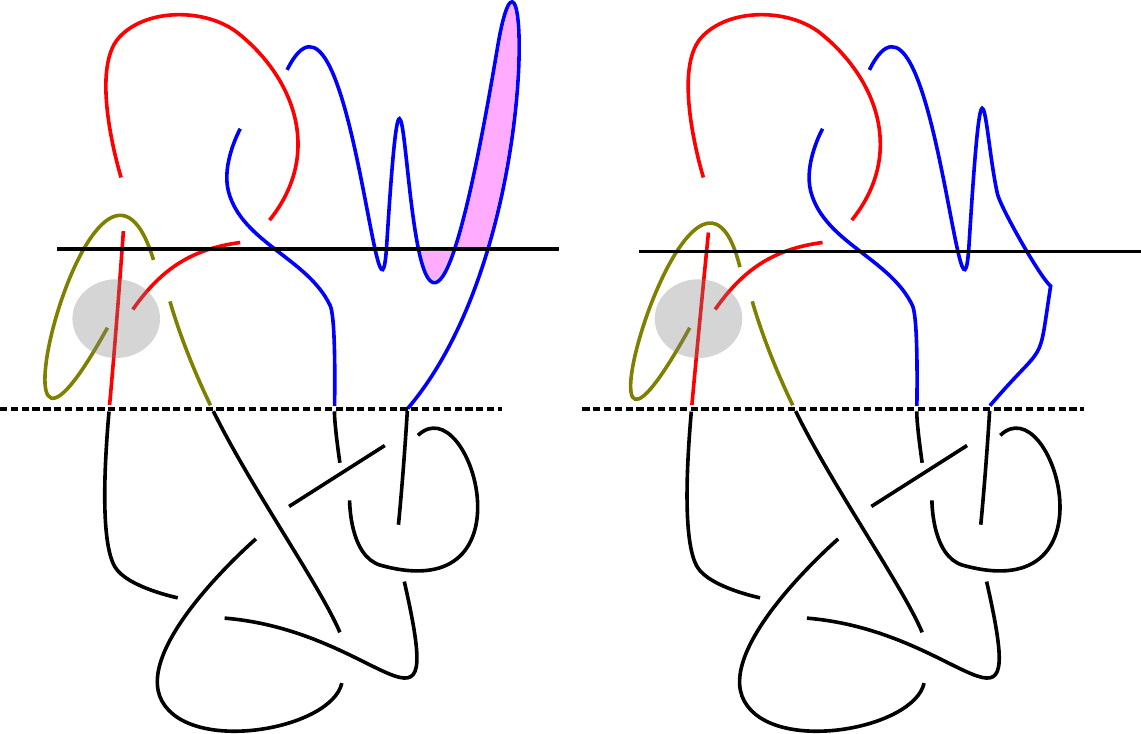}
\caption{There are three local minima, but there is only one special crossing. However, Lemma \ref{lem:nowiggle} shows that if this happens then $L$ is not in thin position with respect to any notion of widths.}\label{fig:decompose}
\end{figure}

From now on, we make the assumption that for any level $h^{-1}(z)$ and any strand $s$, we have $|h^{-1}(z)\cap s|\leq 2$. Notice that the isotopies of the link required to arrange this only require that our embedding is Morse.

\begin{lem}\label{mins}
		Let $(A_0,f_0)\rightarrow \ldots \rightarrow (A_J,f_J)$ be a completed coloring sequence on $D$ obtained via coloring be height. Let $t$ be any stage in the coloring and $r$ a regular value of $h|_L$ such that $h^{-1}(r,\infty)\cap D = A_t$ i.e every strand above $h^{-1}(r)$ has been colored but no strand below has been colored by stage $t$. Then the number of crossings that have become special by stage $t$ is at least the number of minima above $h^{-1}(r)$. 
	\end{lem}
	
	\begin{proof}
		Let $s$ be a strand of $D$ corresponding to a minimum of $h|_L$. Recall $p$ denotes the projection $p(x,y,z)=(0,y,z)$ into the yz-plane. Let $K_s$ be the $S^1$ component of $L$ containing $p^{-1}(s)\cap L$. Say $s\in A_t$ (so $s$ has been colored by stage $t$) and let $D_s$ be the subset of $D$ corresponding to $p(K_s)$.
		
		\textbf{\underline{Case 1:}} Suppose there exists adjacent strands $s_1,s_2,s_3\in D_s$ and a stage $u$ such that $\{s_1,s_3\}\subset A_u$ and $A_{u+1}\setminus A_u=\{s_2\}$ with $f_u(s_p)=f_u(s_r)$ (colloquially, we are saying there are three adjacent strands where the middle strand is the last to get colored, and $u+t\leq t$ is the stage at which $s_2$ gets colored). When coloring by height, only strands with maxima become seeds. Hence in this case, the stated assumption, along with the connectedness of colors and basic Morse theory (i.e. $K_s \cong S^1$ has Euler characteristic zero), implies that one of the two following possibilities arise: either $h|_{K_s}$ has exactly 1 maximum and 1 minimum, or $h|_{K_s}$ has exactly 2 maxima and 2 minima. By the connectedness of colors, this means that once the coloring is completed, either $D_s$ contains one type II special crossing (corresponding to the sub-case where $h|_{K_s}$ has a single maximum and single minimum), or $D_s$ has two type I special crossings (corresponding to the sub-case where $h|_{K_s}$ has exactly 2 maxima and 2 minima).  
		
		\textbf{\underline{Case 1 Sub-case 1:}}
		Suppose all of $D_s$ has been colored by stage $t$. The connectedness of colors implies that $s_2$ must be the final strand of $D_s$ to be colored. Moreover, $s_2$ lies above $h^{-1}(r)$, so that all of $D_s$, and all strands forming over-strands of crossings with strands from $D_s$ as under-strands, must have been colored by stage $t$. Hence, all the special crossings arising from the coloring of $D_s$ have become special by stage $t$.
		
		\textbf{\underline{Case 1 Sub-case 2:}} Now suppose not all of $D_s$ has been colored by stage $t$. By assumption, $D_s$ contains the strand $s$, which corresponds to a minimum that lies above $h^{-1}(r)$. Hence, in this sub-case $h|_{K_s}$ must have 2 minima, and $s$ corresponds the minimum lying above $h^{-1}(r)$ while there is one other minimum lying below $h^{-1}(r)$. We want to show there is one crossing in $D_s$ that has become type I special by stage $t$.
		
		If $s$ received its color via a seed addition, then $h|_s$ contains both a max and and min in the interior of $s$. Since $|h^{-1}(z)\cap s|\leq 2$ for all $z$, then no coloring move could have been performed over the lower crossing of $s$. Hence, the lower crossing becomes type I special by stage $t$.
		
		Now suppose $s$ received its color via a coloring move. Let $x_i$ and $x_j$ denote the crossings adjacent to $s$, with $x_i$ corresponding to the lower endpoint of $s$. Then $s$ must have inherited its color via a coloring move performed over $x_j$. Moreover, since $|h^{-1}(z)\cap s|\leq 2$ for all $z$, no coloring move could have been performed over $x_i$. Since $s$ corresponds to the "higher" minimum of $h|_{K_s}$, we conclude that the strand adjacent to $s$ at $x_i$ must have received a different color than $s$ in the coloring sequence. The assumption that $s$ corresponds to a minimum that lies above $h^{-1}(r)$ implies that the over-strand and the other under-strand of $x_i$ also lies above $h^{-1}(r)$. Hence, $x_i$ becomes type I special by stage $t$.
		
		\textbf{\underline{Case 2:}} Suppose there are no three adjacent strands where the middle is the last to get colored. We proceed using the same reasoning as in the converse direction of Proposition 5.5 part (2) in \cite{lee2019algorithmic} to show that the crossing corresponding to the lower endpoint of $s$ becomes special by stage $t$.
		
		Let $x_i$ and $x_j$ be the crossings corresponding to the endpoints of $s$, with $x_i$ being at the lower endpoint. We claim $x_i$ becomes special by stage $t$.
		
		Let $s_i$ and $s_j$ be the strands adjacent to $s$ at $x_i$ and $x_j$, respectively. Let $s_k$ denote the other strand adjacent to $s_i$. Let $u<t$ be the stage at which $s$ receives its color. Note that since we colored by height and $h|_s$ has a minimum in the interior of $s$, the over-strands of $x_i$ and $x_j$, in addition to $s_i,s_j$ and $s_k$, have all been colored by stage $t$. Suppose for contradiction $x_i$ is not special, meaning that $s_i$ and $s$ have received the same color by stage $t$. 
		
		\textbf{\underline{Sub-case 1:}} In this sub-case, we suppose that $s$ received its color via a seed addition. Since $s_i$ and $s$ have the same color and $s$ is the seed corresponding to its color, $h|_{s_i}$ cannot have a maximum in the interior of $s_i$, as this would mean $s_i$ would be another seed strand, necessarily of a different color than $s$. Since $\text{max}(|h^{-1}(z)\cap s_i|,|h^{-1}(z)\cap s|)\leq 2$ for all $z$, $h|_s$ has a minimum in its interior, and $x_i$ is the lower crossing of $s$, then $s_i$ must be monotonic with respect to $h$. Hence $s$ inherits it's color from $s_k$ via a coloring move $(A_l,f_l)\rightarrow (A_{l+1},f_{l+1})$. But then $\{s_k,s\}\subset A_l$, contradicting the initial assumption of case 2 as $s_k,s_i,$ and $s$ are three adjacent strands.
		
		\textbf{\underline{Sub-case 2:}} In this sub-case, we suppose that $s$ received its color via a coloring move $(A_u,f_u)\rightarrow (A_{u+1},f_{u+1})$. Since $x_i$ is the lower crossing of $s$, this coloring move must have been performed over $x_j$. By assumption, of the three adjacent strands $s_i,s$ and $s_j$, $s$ cannot be the last of the three to be colored. Hence, $s_i\notin A_{u+1}$.
		
		This means there exists $u<l<t$ such that $A_{l+1}\setminus A_l=\{s_i\}$. Moreover, $(A_l,f_l)\rightarrow (A_{l+1},f_{l+1})$ must have been a coloring move since $s_i$ and $s$ receive the same color, on account of $x_i$ not being special. No coloring move could have been performed over $x_i$ since $h|_s$ has a minimum in its interior and $x_i$ is the lower crossing of $s$. Therefore, $s_i$ inherits its color from $s_k$, meaning $\{s_k,s\}\subset A_l$, a contradiction to the initial assumption of case 2 as $s_k,s_i$ and $s$ are three adjacent strands.	
	\end{proof}

\begin{prop}\label{prop:colorbyheight}
		For any link $L$, $\mathbb{WL}(L)$ lower bounds the lex-width of $L$. 
	\end{prop}
	
	\begin{proof}
		 Let $(A_0,f_0)\rightarrow \ldots \rightarrow (A_J,f_J)$ be a completed coloring sequence of $D$ obtained via coloring by height. Let $(a_i)_{i=0}^N$ denote the sequence associated to the coloring, indexed in the order dictated by the stage of the coloring in which they occur. Fix any $a_k$ in the associated sequence. Let $(A_0,f_0)\rightarrow \ldots \rightarrow (A_t,f_t)$ be the partial coloring sequence where $t$ is the first stage at which $a_k$ appears. 
		
		Since we color by height, there exists a regular level $r$ where $h^{-1}((\infty, r))\cap D=A_t$, i.e. a regular level surface where all the strands above $r$ have been colored but no strand below has been colored, by stage $t$. The proof of this Proposition is completed by verifying the following Claim:

			\textbf{Claim:} $a_k\leq |h^{-1}(r)\cap L|$.

		Put $a_{-1}=0$ and write 
		\begin{equation}\label{wirtsum}
		a_k=\sum_{i=0}^{k}a_i-a_{i-1}.
		\end{equation}
		Let $h^{-1}(r_0),\ldots ,h^{-1}(r_m)$ be level surfaces, listed in order of decreasing height, such that we have exactly one level between each adjacent pair of critical points above $h^{-1}(r)$. Moreover, let $h^{-1}(r_m)=r^{-1}(r)$. Put $|h^{-1}(r_{-1})\cap L|=0$ and write 
		\begin{equation}\label{trunksum}
		|h^{-1}(r)\cap L| = \sum_{i=0}^m|h^{-1}(r_i)\cap L| - |h^{-1}(r_{i-1})\cap L|.
		\end{equation}
		
		Notice that the equations on the right hand side of (\ref{wirtsum}) and (\ref{trunksum}) are a finite sums of $+2$'s and $-2$'s. In the summation for equation (\ref{trunksum}), there is a single +2 for every maximum above $h^{-1}(r)$, and a single -2 for every minimum above $h^{-1}(r)$. 
		
		In equation (\ref{wirtsum}), it follows from our definition of our completed coloring sequence that there is a single $+2$ for every strand $s$ with height greater than $r$ where $h|_s$ achieves its maximum in the interior of $s$. That means the number of positive terms in the summation for (\ref{wirtsum}) is bounded above by the number of positive terms in the summation for (\ref{trunksum}).
		
		The number of $-2$'s in the summation of equation (\ref{wirtsum}) is equal to the number of crossings that have become special by stage $t$. By Lemma \ref{mins}, every minimum above $h^{-1}(r)$ must have resulted in a special crossing in our partial coloring sequence. Therefore, the summation in equation (\ref{wirtsum}) has at least as many negative terms as the summation in equation (\ref{trunksum}). This gives the desired claim. 	
	\end{proof}
	
\begin{cor}
    For any link $L$, $\mathbb{WT}(L) \leq trunk(L)$ and $\mathbb{WS}(L) \leq w(L)$. 
\end{cor}

\section{Embedding from link diagram coloring}\label{sec:embeddingfromcoloring}

\main*

\begin{proof}
The Wirtinger lex-width was shown to be less than the lex-width in Proposition \ref{prop:colorbyheight}. We now show the other direction of the inequality. We prove the theorem by induction on $m$, the number of components of $L$.

Suppose there are $N$ strands total. Embed each strand in the plane $y = c$ in 3-space, where $c$ is a constant determined by the coloring sequence in decreasing height. More precisely, say the beginning seeds are denoted $s_1,...,s_k$. Then, embed $s_i (i=1,...,k)$ in the plane $y= -i.$ After all the beginning seeds are embedded, we embed the next strands that receive the colors next in order (this can be from a coloring move or seed addition) in the plane $y = -(k+1), y=-(k+2),...$ and so on until the final strand in the coloring sequence is embedded on $y=-N.$ To be more concrete, the purple seed strand from Figure \ref{fig:wirtwidth} gets embedded in level $y = -9$.

At the moment, we have disconnected strands floating in 3-space. We must connect them up to form a link. Furthermore, we must show that the link type agrees with the diagram it was lifted from. Denote the overstrand at a crossing $x$ as $a_k,$ and the understrands as $a_i$ and $a_j.$ We now divide into two main cases depending on whether $L$ is a knot or a link.\\

\textbf{\underline{Case 1:}} $L$ has 1 component so that $L$ is a knot.\\

\textbf{\underline{Case 1 Sub-case 1:}} 
Suppose that $a_i$ and $a_j$ receive the same color at the end of the coloring process. Since $L$ is a knot, this means the overstrand $a_k$ cannot be colored last among $\{a_i,a_j,a_k\}$. Because if so, one could trace out a link component, which contradicts the assumption that $L$ is a knot.

We can then connect an endpoint of $a_i$ to an endpoint of $a_j$ by a monotonic arc so that when we project downward, we get the crossing $x$ in the link diagram. This creates no new maxima or minima, and this can be done because $a_k$ is already higher than both $a_i$ and $a_j$ with respect to the height function.

\textbf{\underline{Case 1 Sub-case 2:}} Suppose that $a_k$ is embedded in a plane lower than $a_i$ and $a_j$. We connect an endpoint of $a_i$ to an endpoint of $a_j$ by an arc that creates a minimum so that when we project downward, we get a multi-crossing $x$ in the link diagram. 

Each seed can be perturbed to give a maximum. The process so far is like the ones in \cite{blair2020wirtinger,pongtanapaisan2019wirtinger}. But now, the difference pointed out in \cite{lee2019algorithmic} is if a special crossing $x$ appears before a seed addition $s$, then we embed the minimum corresponding to $x$ higher than the maximum corresponding to $s.$\\

\textbf{\underline{Case 2:}} $L$ has more than 1 components.\\
Assume that the Wirtinger width of $L'$ equals its width, where $L'$ is a link of fewer than $m$ components. First, we consider the case that $L$ has a diagram with realizing the Wirtinger width, where $D$ is
not cut-split (see Subsection \ref{subsec:cutsplit}). 

\textbf{\underline{Case 2 Sub-case 1:}} 
In this case, we can have a phenomenon not present in Case 1. Namely, it is possible for $a_i$ and $a_j$ to receive the same color, but $a_k$ is colored after $a_i$ and $a_j.$ Say, this happens at the crossing $x$. In this case, we get an unknot component $U$ traced out by one color, which is a special crossing of type II. As pointed out in \cite{blair2020wirtinger}, we get a local minimum correspond to this color. Now, if the strand $a_k$ corresponding to $x$ gets colored before a seed addition $s,$ then we embed the minimum corresponding to $x$ higher than the maximum corresponding to $s.$

If $a_i$ and $a_j$ receive the same color at the end of the coloring process, but the height of $a_k$ is higher than the minimum heights among $\{a_i,a_j\},$ we can still make a connection that creates no new local maximum or minimum just as Case 1 Subcase 1. On the other hand, if $a_k$ is embedded in a plane lower than $a_i$ and $a_j$, where $a_i$ and $a_j$ received different colors, then this gives a local minimum like Case 1 Subcase 2. 

\textbf{\underline{Case 2 Sub-case 2:}} 

Assume both $D$ and $L$
are cut-split with splitting component $U$. That is, $U$ projects to a self-adjacent strand or a simple closed curve in $s(D)$. Let $D'$ be a diagram obtained by removing the self-adjacent strand or a simple closed curve.

The statement follows from the following string of inequalities. The first equality comes from observing that adding a splitting component changes the sequence by adding two particular numbers into the associated sequence. The second equality comes from the inductive hypothesis, where we have one fewer component. The final inequality can be seen by noticing that such an associated sequence from the coloring with two fewer terms of $L\backslash U$ once a splitting component is removed exists. 

\begin{align*}
    \text{lex-width}(L) &= \text{lex-width}(L\backslash U) \cup \{a_i,a_{i}-2\} \\ &= \mathbb{WL}(L\backslash U) \cup \{a_i,a_{i}-2\} \\ &\leq \mathbb{WL}(L).
\end{align*}

\end{proof}

\begin{cor}
    For any link $L$, $\mathbb{WT}(L) = trunk(L)$ and $\mathbb{WS}(L) = w(L)$. 
\end{cor}

\section{Computational results}\label{sec:computational}

In \cite{lee2019algorithmic}, the author's Python code works for knots (i.e. 1-component links) admitting a diagram with Wirtinger number 4. In this paper, we modify the code so that we can compute various notions of widths for multi-component links with Wirtinger number 4. A new challenge include adapting the code to detect the Type II special crossing, which is not present in the 1-component setting. Our code can be found in \cite{Code}.

The Gauss code of a knot can be thought of as a list of numbers. On the other hand, a multi-component link's Gauss code is made up of multiple lists. We begin by using the code in \cite{blair2020wirtinger} to find the Wirtinger numbers of links up to 14 crossings. We converted from DT codes of links up to 14 crossings from SnapPy \cite{SnapPy} to Gauss codes. If the Wirtinger algorithm outputs 2, then we know it is a 2-bridge link, and the trunk is 4. When the Wirtinger algorithm outputs a number greater than two, it is not guaranteed that the Wirtinger number (i.e., the bridge number) cannot be lower. To ensure this, we use the method from \cite{blair2022coxeter} to find matching lower bounds.

If the lower bound matches the upper bound of 3, then the trunk is 6 we took the links from the table of prime links. If the trunk were 4, then there must be a decomposing sphere, which only exists for composite links.

Roughly, the code works as follows. The code selects a subset of three strands, and extend the colorings as much as possible using only the coloring moves. Since we assumed that the Wirtinger number is 4, we know that the entire diagram is not colored.

If a special crossing is found, then we select another strand that we didn't select as the three initial seeds, and perform a seed addition move using this new strand. If these four strands together can be used to color the entire diagram using only the coloring moves, then the trunk is 6.

We remark that if the code outputs 8 for a link $L$ from the table, then we know that trunk$(L)$ is 6 or 8. A lower bound would need to be computed in order to actually show that the trunk is exactly 8. We did not pursue this direction for this current work, but properties such as waist \cite{ozawa2009waist}, representativity \cite{blair2019height}, and mp-smallness \cite{ozawa2009waist} can be used as a lower estimate. Nonetheless, the links where our program returns 8 is a good candidate for a link with no meridionally essential planar surfaces.

\section{Applications and questions}\label{sec:apps}
\subsection{Tube embeddings}

It is useful to model polymers under confinement as self-avoiding polygons in the simple cubic lattice \cite{hammersley1985self}. For $n,m \in \mathbb{Z}$, an \textit{$(m\times n)$-tube} is the subset $\mathbb{R}\times [0,m]\times[0,n].$ In \cite{ishihara2017bounds}, it is shown that embedding in a tube is characterized by the trunk. Links in the $(2\times 1)$-tube are more well-understood. Namely, a link type
$K$ can be confined in the $(2
\times 1)$-tube if and only if $K$ is a trivial knot or link, a 2-bridge knot or link, or connected sum or split sum of 2-bridge knots and links. This fact is an ingredient that the authors in \cite{beaton2022entanglement} used to answer difficult questions related to exponential growth rates. 

Toward a better understanding of linking statistics in larger tube sizes, we can utilize the current work to identify numerous specific links that fit in a $(3\times 1)$-tube.

\begin{cor}
    The 4-bridge links whose width is detected by our program fit in a $(3\times 1)$-tube.
\end{cor}

\subsection{Thin position of branched cover}

The width of a 3-manifold \cite{scharlenaann1994thin} is a multi-set of integers, which record the information about the genera of the thick and thin levels.

More precisely, any closed 3-manifold $M$ can be constructed in stages $M = b_0 \cup N_1 \cup T_1 \cup N_2 \cup T_2 \cup \cdots \cup N_k \cup T_k \cup b_3$, where
$b_0$ is a collection of 0-handles, $b_3$ is a collection of 3-handles, each $N_i$ consists of 1-handles, and each $T_i$
consists of 2-handles $(1\leq i\leq k)$. Denote by $S_i$ the surface obtained from $\partial[b_0 \cup N_1 \cup T_1 \cup N_2 \cup T_2 \cup \cdots \cup N_i]$ by deleting all spheres bounding 0-handles or 3-handles in the decomposition.

By the \textit{complexity} $c(S)$ of a connected surface $S$, we mean
\begin{center}
    $c(S) = \begin{cases} 1 -\chi(S) & \text{if $S$ has positive genus} \\ 0 & \text{if $S$ is a 2-sphere}.\end{cases}$
\end{center}

The complexity of a disconnected surface is the sum of the complexities over the connected components.
\begin{rem}
    A reason why such a function is needed rather than merely recording the Euler characteristics of the surfaces is so that the function behaves well under appropriate operations used to prove theorems of interest.
\end{rem}
\begin{defn}[\cite{scharlenaann1994thin}]
    The \textit{width} of the decomposition of $M$ be the multi-set $\{c(S_i) | 1  \leq i  \leq k\}$. The \textit{width} of a 3-manifold $M$, denoted $w(M)$, is the minimal width over all decompositions using the dictionary
ordering (i.e., order the integers in each multi-set monotonically in non-increasing order, and then compare the ordered multi-sets lexicographically).
\end{defn}

As pointed out in \cite{howards2008thin} by Howards and Schultens, the width of the double branched cover $M_2(L)$ of $S^3$ along $L$ is related to the lex-width of $L.$ More specifically, in the multi-set realizing the lex-width of $L$, we begin by removing all the entries except the ones coming from the thick levels. Suppose that the resulting multi-set is $\lbrace a_1,a_2,\cdots, a_k\rbrace.$ Then, the sequence of the genera of thick levels are $\lbrace \frac{a_1}{2}-1,\frac{a_2}{2}-1,\cdots, \frac{a_k}{2}-1\rbrace$ and the width of $M_2(L)$ is at most $\lbrace a_1-3,a_2-3,\cdots, a_k-3\rbrace.$ A consequence of the computations done in this paper is the following.

\begin{cor}
    The width of the double branched covers of the 4-bridge links whose width is detected by our program is $\{3,3\}.$ 
\end{cor}
\subsection{Thin position of graphs}
As spatial graph theory can be thought of as a generalization of knot theory, it is natural to study the width of spatial graphs. Naively, perhaps one may want to define the width in terms of the intersection of thick and thin levels with the graph itself. Scharlemann pointed out, however, that by defining the width this way, we do not get some of the desirable conclusions as in the knots and links case \cite{scharlemann2005thin}. Instead, it is more helpful to define the width of a spatial graph by considering certain links associated to the graph.

There is a convenient way to construct a link $L_{\Gamma}$ from a spatial graph $\Gamma$ by letting $L_{\Gamma}$ be the boundary of a ribbon surface whose core is $\Gamma$. In \cite{scharlemann2005thin,li2010thin}, the author showed that the Gabai width of the link $L_{\Gamma}$ gives useful information about the essential surfaces in the exterior of $\Gamma.$

\subsection{Questions}

There are several conjectures related to Morse positions of knots that can be rephrased using our results. This may offer new perspectives.

\begin{defn}
    The \textit{height} of a conformation $\gamma$ of a knot $K$, denoted $ht(\gamma)$, is the number of thick level spheres for $\gamma.$ The \textit{height} $ht(K)$ (resp. \textit{min-height} $ht_{\min}(K)$) of a knot is the maximum (resp. minimum) height $ht(\gamma)$  over all thin positions of $K$. 
\end{defn}
The following conjecture is expected.
\begin{conj}
    There is a knot $K$ such that $ht_{\min}(K)<ht(K)$.
\end{conj}
Call a consecutive seed addition moves with no special crossings in between a \textit{chains of seeds}. Using our main theorem, we see that the height of $\gamma$ is the number of chains of seeds.

\begin{quest}
Is the maximum number of chains over all diagrams realizing Wirtinger width the same as the minimum number of chains over all diagrams realizing Wirtinger width?
\end{quest}

\subsection*{Acknowledgements}
We thank Ryan Blair, Jeremy Eng, Rob Scharein, and Chris Soteros for helpful conversations.

\bibliographystyle{plain}
\bibliography{ref}
\end{document}